\documentclass[graybox]{svmult}

\usepackage{mathptmx}       
\usepackage{helvet}         
\usepackage{courier}        
\usepackage{type1cm}        
\usepackage{makeidx}         
\usepackage{graphicx}        
\usepackage{epstopdf}

\usepackage{multicol}        
\usepackage[bottom]{footmisc}

\usepackage{bbold}
\usepackage{amssymb}
\usepackage{amsmath}
\usepackage{amsfonts}
\usepackage{mathbbol}
\usepackage{setspace}
\usepackage{multirow}

\newcommand{\href}[2]{#2}
\newcommand{\arxiv}[1]{\href{http://www.arXiv.org/abs/#1}{arXiv:#1}}

\newcommand{\R}{\mathbb{R}}
\newcommand{\C}{\mathbb{C}}

\newcommand{\diag}{\operatorname{diag}}
\newcommand{\rmax}{\R_{\max}}
\newcommand{\rmaxm}{\R_{\max,\times}}
\newcommand{\trop}{\mathsf{t}}
\newcommand{\zero}{\mathbb{0}}
\newcommand{\unit}{\mathbb{1}}

\newcommand{\set}[2]{\{#1\mid #2\}}
\newcommand{\NEW}[1]{{\em #1}\index{#1}}
\newcommand{\cond}{\operatorname{cond}}
\newcommand{\spec}{\operatorname{spec}}

\makeindex             

\begin{document}

\title*{Tropical Scaling of Polynomial Matrices}
\author{St\'ephane Gaubert\and Meisam Sharify}

\institute{St\'ephane Gaubert \at INRIA Saclay -- \^Ile-de-France \& Centre de Math\'ematiques appliqu\'ees, Ecole Polytechnique,
91128 Palaiseau, France, \email{Stephane.Gaubert@inria.fr}
\and Meisam Sharify \at INRIA Saclay -- \^Ile-de-France \& Centre de Math\'ematiques appliqu\'ees, Ecole Polytechnique,
91128 Palaiseau, France, \email{Meisam.Sharify@inria.fr}}
%
%
\maketitle

\abstract*{The eigenvalues of a matrix polynomial can be determined classically
by solving a generalized eigenproblem for a linearized matrix pencil,
for instance by writing the matrix polynomial in companion form.
We introduce a general scaling technique, based on tropical algebra, 
which applies in particular to this companion form. This scaling,
which is inspired by an earlier work of Akian, Bapat, and Gaubert,
relies on the computation of ``tropical roots''. We give explicit
bounds, in a typical case, indicating that these roots provide accurate estimates of the order
of magnitude of the different eigenvalues, and we show by experiments
that this scaling improves the accuracy (measured by
normwise backward error) of the computations, particularly in situations
in which the data have various orders of magnitude.
In the case of quadratic polynomial matrices, we recover in this way a
scaling due to Fan, Lin, and Van Dooren, which
coincides with the tropical scaling when the two tropical roots are equal.
If not, the eigenvalues generally split in two groups, and the tropical
method leads to making one specific scaling for each
of the groups.}

\abstract{The eigenvalues of a matrix polynomial can be determined classically
by solving a generalized eigenproblem for a linearized matrix pencil,
for instance by writing the matrix polynomial in companion form.
We introduce a general scaling technique, based on tropical algebra, 
which applies in particular to this companion form. This scaling,
which is inspired by an earlier work of Akian, Bapat, and Gaubert,
relies on the computation of ``tropical roots''. We give explicit
bounds, in a typical case, indicating that these roots provide accurate estimates of the order
of magnitude of the different eigenvalues, and we show by experiments
that this scaling improves the accuracy (measured by
normwise backward error) of the computations, particularly in situations
in which the data have various orders of magnitude.
In the case of quadratic polynomial matrices, we recover in this way a
scaling due to Fan, Lin, and Van Dooren, which
coincides with the tropical scaling when the two tropical roots are equal.
If not, the eigenvalues generally split in two groups, and the tropical
method leads to making one specific scaling for each
of the groups.}

\section{Introduction}
\label{sec:1}
A classical problem is to compute the eigenvalues of a matrix polynomial
\[
P(\lambda)=A_{0}+A_{1}\lambda+\cdots +A_{d}\lambda^d
\]
where $A_{l}\in \C^{n\times n},l=0 \ldots d$ are given. The eigenvalues are
defined as the solutions of $\det(P(\lambda))=0$. If $\lambda$ is an eigenvalue, the
associated right and left eigenvectors $x$ and $y\in \C^n$ are the non-zero
solutions of the systems $P(\lambda)x=0$ and $y^*P(\lambda)=0$, respectively.
A common way to solve this problem, is to convert $P$ into a 
``linearized'' matrix pencil
\[
L(\lambda)=\lambda X+Y,\quad X,Y\in\C^{nd\times nd}
\]
with the same spectrum as $P$ and solve the eigenproblem for $L$,
by standard numerical algorithms like the QZ method~\cite{moler73}.
If $D$ and $D'$ are invertible diagonal matrices, and
if $\alpha$ is a non-zero scalar, we may consider
equivalently the scaled pencil $DL(\alpha\lambda)D'$.

The problem of finding the good linearizations and the good scalings
has received a considerable attention. 
The backward error and conditioning of the matrix pencil problem and 
of its linearizations have been investigated in particular
in works of Tisseur, Li, Higham, and Mackey,
 see~\cite{tisseur00,Higham061,Higham062}.

A scaling on the eigenvalue parameter to improve the normwise 
backward error of a quadratic polynomial matrix was proposed by
Fan, Lin, and Van Dooren~\cite{vandooren00}. This scaling
only relies on the norms $\gamma_l:=\|A_l\|$, $l=0,1,2$.
In this paper, we introduce a new family of scalings which also
rely on these norms. The degree $d$ is now arbitrary.

These scalings originate from the work of Akian, Bapat, and Gaubert~\cite{abg04b,abg04}, in which the entries of the matrices $A_l$ are functions, for instance
Puiseux series, of a (perturbation) parameter $t$. The valuations (leading exponents) of the Puiseux series representing the different eigenvalues
were shown to coincide, under some genericity conditions,
with the points of non-differentiability
of the value function of a parametric optimal assignment problem
(the tropical eigenvalues), a result which can be interpreted in
terms of amoebas~\cite{itenberg}. Indeed, the definition of the tropical
eigenvalues in~\cite{abg04b,abg04} makes sense in any field
with valuation. In particular, when the coefficients
belong to $\C$, we can take the map
$z\mapsto \log |z|$ from $\C$ to $\R\cup\{-\infty\}$
as the valuation. Then, the tropical eigenvalues are
expected to give, again under some non degeneracy conditions, the
correct order of magnitude of the different eigenvalues. 

The tropical
roots used in the present paper are an approximation
of the tropical eigenvalues, relying only on the norms
$\gamma_l=\|A_l\|$. A better scaling may be achieved
by considering the tropical eigenvalues, but computing these
eigenvalues requires $O(nd)$ calls to an optimal assignment algorithm,
whereas the tropical roots considered here can be computed in $O(d)$
time, see Remark~\ref{rk-eig} below for more information.
We examine such extensions in a further work.

As an illustration, consider the following quadratic polynomial matrix
\[
P(\lambda)=\lambda^210^{-18} \begin{pmatrix}1&2\\3&4\end{pmatrix}+\lambda\begin{pmatrix}-3&10\\16&45\end{pmatrix}+10^{-18}\begin{pmatrix}12&15\\34&28\end{pmatrix}
\]

By applying the QZ algorithm
on the first companion form of $P(\lambda)$ we get
the eigenvalues  -Inf,- 7.731e-19 , Inf, 3.588e-19,
by using the scaling proposed in ~\cite{vandooren00} we get 
-Inf, -3.250e-19, Inf, 3.588e-19.
However by using the tropical scaling we can find 
the four eigenvalues properly:
- 7.250e-18 $\pm$  9.744e-18i, - 2.102e+17 $\pm$ 7.387e+17i.
The result was shown to be correct (actually, up to a 14 digits precision)
with PARI, in which an arbitrarily large precision can be set. 
The above computations were performed in Matlab (version 7.3.0).

The paper is organized as follows.
In Section~\ref{sec:2}, we recall some classical facts of max-plus or tropical algebra,
and show that the tropical roots of a 
tropical polynomial can be computed in linear time, 
using a convex hull algorithm. 
Section~\ref{sec:3} states preliminary results concerning matrix pencils, linearization and normwise backward error.
 
In Section~\ref{sec:4}, we describe our scaling method.
In Section 5, we give a theorem locating the eigenvalues
of a quadratic polynomial matrix, which
provides some theoretical justification of the method. 
Finally in Section~6, we present the experimental results
showing that the tropical scaling can highly reduce the normwise backward error of an eigenpair. We consider the quadratic case in Section~6.1 and the general case in Section~6.2. For the quadratic case, we compare our results with the scaling proposed in ~\cite{vandooren00}.

\section{Tropical polynomials}
\label{sec:2}
The max-plus semiring $\rmax$, is the set $ \R\cup\{-\infty\} $, equipped with max as addition, and the usual addition as multiplication. 
It is traditional to use the notation $\oplus$ for  $\max$
(so $2 \oplus 3 = 3$), and $\otimes$ for $+$ (so $1\otimes 1 = 2$).
We denote by $\zero$ the zero element of the semiring, which is such that $\zero\oplus a = a$, here $\zero=-\infty$, and by $\unit$ the unit element of the semiring, which is such that $\unit\otimes a = a \otimes \unit = a$, here $\unit = 0$.
We refer the reader to~\cite{BCOQ92,KM,AG} for more background.

A variant of this semiring is 
the max-times semiring $\rmaxm$, 
which is the set of nonnegative real numbers $\R^+$, equipped with max as addition, and $\times$ as
multiplication. This semiring is isomorphic to $\rmax$ by the map
$x\mapsto \log x$. So, every notion
defined over $\rmax$ has an $\rmaxm$ analogue that we shall not
redefine explicitly. In the sequel, the word ``tropical'' will
refer indifferently to any of these algebraic structures.

Consider a max-plus (formal)
polynomial of degree $n$ in one variable, i.e., a formal
expression $P=\bigoplus_{0\leq k\leq n}P_kX^k$ in which the coefficients $P_k$
belong to $\rmax$, and the
associated numerical polynomial, which, with the notation
of the classical algebra, can be written as $p(x)=\max_{0\leq k\leq n}P_k+kx$.
Cuninghame-Green and Meijer showed~\cite{cuning80} that the analogue of the fundamental theorem of algebra holds in the max-plus setting, i.e.,
that $p(x)$ can be written uniquely as 
$p(x)=P_{n}+\sum_{1\leq k\leq n} \max(x,c_k)$, 
where $c_1,\ldots,c_n\in\rmax$ are the \NEW{roots}, i.e., the points at which the maximum attained at least twice. This is a special case of more general notions
which have arisen recently
in tropical geometry~\cite{itenberg}.
The \NEW{multiplicity} of the root $c$ is the cardinality of the set 
$\set{k\in \{1,\ldots, n\}}{c_{k} = c}$. Define the {\em Newton polygon}
$\Delta(P)$ 
of $P$ to be the upper boundary of the convex hull of the set of points
$(k,P_k)$, $k=0,\ldots,n$. This boundary consists of a
number of linear segments.
An application of Legendre-Fenchel duality
(see~\cite[Proposition 2.10]{abg04b}) shows that the 
opposite of the slopes of these segments are precisely the tropical
roots, and that the multiplicity of a root coincides with the horizontal
width of the corresponding segment. (Actually, min-plus polynomials
are considered in~\cite{abg04b}, but the max-plus case reduces
to the min-plus case by an obvious change of variable).
Since the Graham scan algorithm~\cite{Graham72} allows us to compute the convex hull of a finite set of points by making $O(n)$ arithmetical operations and comparisons, provided that the given set of points is already sorted by abscissa, we get the following result.
\begin{proposition}
The roots of a max-plus polynomial in one variable can be computed
in linear time.
\qed
\label{prp:troprootcomp}
\end{proposition}
The case of a max-times polynomial reduces
to the max-plus case by replacing every coefficient by its logarithm.
The exponentials of the roots of the transformed polynomial are
the roots of the original polynomial.

\section{Matrix pencil and normwise backward error}
\label{sec:3} 
Let us come back to the eigenvalue problem for the matrix pencil 
$P(\lambda)=A_{0}+A_{1}\lambda+\cdots +A_{d}\lambda^d$.\ 
There are many ways to construct a ``linearized'' matrix pencil
$L(\lambda)=\lambda X+Y,\quad X,Y\in\C^{nd\times nd}$ with the same
spectrum as $P(\lambda)$, see~\cite{Mackey06} for a general discussion.
In particular, the first 
companion form $\lambda X_1 + Y_1$
is defined by
\[
X_1 
= \diag(A_k,I_{(k-1)n}),
\qquad 
Y_1=
\begin{pmatrix}
  A_{k-1}&A_{k-2}&\ldots&A_0\\
    -I_n & 0 &\ldots &0 \\
   \vdots & \vdots &\vdots & \ddots \\
  0 & \ldots & -I_n & 0 
\end{pmatrix} \enspace .
\]

In the experimental part of this work, we are using this linearization.

To estimate the accuracy of a numerical algorithm computing an eigenpair,
we shall consider, as in~\cite{tisseur00}, the normwise backward error.
The latter arises when considering a perturbation
\[
\Delta P= \Delta A_{0}+\Delta A_{1}\lambda+\cdots +\Delta A_{d}\lambda^d \enspace. 
\]
The backward error of an approximate eigenpair
 $(\tilde{x},\tilde{\lambda})$ of $P$ 
is defined by
\[
\eta(\tilde{x},\tilde{\lambda})=\min\{\epsilon:(P(\tilde{\lambda})+\Delta P(\tilde{\lambda}))\tilde{x}=0,\|\Delta A_{l}\|_2\leq\epsilon\|E_l\|_2,l=0,\ldots m\} \enspace .
\]
The matrices $E_l$ representing tolerances.
The following computable expression for $\eta(\tilde{x},\tilde{\lambda})$ is given in the same reference,
\[
\eta(\tilde{x},\tilde{\lambda})=\frac{\|r\|_{2}}{\tilde{\alpha}\|\tilde{x}\|_{2}}
\]
where $r=P(\tilde{\lambda})\tilde{x}$ and $\tilde{\alpha}=\sum{|\tilde{\lambda}|^l\|E_l\|_2}$. In the sequel, we shall take $E_l=A_l$.

Our aim is to reduce the normwise backward error, by
a scaling of the eigenvalue $\lambda=\alpha\mu$,
where $\alpha$ is the scaling parameter.
This kind of scaling for quadratic polynomial matrix was proposed by Fan, Lin and Van Dooren~\cite{vandooren00}. We next introduce a new scaling,
based on the tropical roots.

\section{Construction of the tropical scaling}
\label{sec:4} 
Consider the matrix pencil modified by the substitution $\lambda=\alpha\mu$
\[
\tilde{P}(\mu)=\tilde{A}_{0}+\tilde{A}_{1}\mu+\cdots +\tilde{A}_{d}\mu^d
\]
where $\tilde{A}_{i}=\beta\alpha^i{A}_{i}$.

The tropical scaling which we next introduce is characterized by the property
that $\alpha$ and $\beta$ are such that $\tilde{P}(\mu)$ has at least two matrices $\tilde{A}_{i}$ with an (induced) Euclidean norm equal to one, whereas the Euclidean norm of the other matrices are all bounded by one.
This scaling is inspired by the work of 
M. Akian and R. Bapat and S. Gaubert~\cite{abg04}, which concerns the perturbation of the eigenvalues of a matrix pencil. The theorem on the location of the eigenvalues which is stated in the next section provides some justification for the
present scaling.

We associate to the original pencil the max-times
polynomial
\[
\trop p(x)=\max(\gamma_{0},\gamma_{1} \lambda,\cdots,\gamma_{d} \lambda^d) \enspace,
\]
where 
\[ \gamma_{i}:=\|A_{i}\|
\]
(the symbol $\trop$ stands for ``tropical'').
Let $\alpha_{1}\leq \alpha_{2}\leq \ldots\leq\alpha_{d}$ be the tropical roots of $\trop p(x)$ counted with multiplicities. 
For each $\alpha_{i}$, the maximum is attained by at least two mononomials.
Subsequently, 
the transformed polynomial $q(x):=\beta_i \trop p(\alpha_i x)$,
with $\beta_i:=(\trop p(\alpha_{i}))^{-1}$ has two coefficients
of modulus one, and all the other coefficients have modulus 
less than or equal to one. 
Thus $\alpha=\alpha_{i}$ and $\beta=\beta_i$ will satisfy the goal. 

The idea is to apply this scaling for all the tropical roots of $\trop p(x)$ and each time, to compute $n$ out of $nd$ eigenvalues of the corresponding scaled matrix pencil, because replacing $P(\lambda)$ by $P(\alpha_i\mu)$ is expected
to decrease the backward error for the eigenvalues of order $\alpha_i$, while 
possibly increasing the backward error for the other ones.

More precisely, let $\alpha_{1}\leq \alpha_{1}\leq \ldots\leq\alpha_{d}$ denote the tropical roots of $\trop p(x)$.
Also let 
\[
\underbrace{\mu_{1},\ldots,\mu_n},\underbrace{\mu_{n+1},\ldots,\mu_{2n}},\ldots,\underbrace{\mu_{(d-1)n+1},\ldots,\mu_{nd}}
\]
be the eigenvalues of $\tilde{P}(\mu)$ sorted by increasing modulus, computed by setting $\alpha=\alpha_{i}$ and $\beta=\trop p(\alpha_{i})^{-1}$ and partitioned in $d$ different groups.
Now, we choose the $i$th group of $n$ eigenvalues, multiply by $\alpha_{i}$ and put in the list of computed eigenvalues. By applying this iteration for all $i=1\ldots d$, we will get the list of the eigenvalues of $P(\lambda)$. Taking into account this description, we arrive at Algorithm~1. It should be understood here
that in the sequence $\mu_1,\ldots,\mu_{nd}$ of eigenvalues above, only the eigenvalues of order $\alpha_i$ are hoped to be computed accurately. Indeed,
in some extreme cases in which the tropical roots have very different orders of magnitude (as in the example shown in the introduction), the eigenvalues of
order $\alpha_i$ turn out to be accurate whereas the groups of higher orders have some eigenvalues Inf or Nan. So, Algorithm~1 merges into
a single picture several snapshots of the spectrum, each of them being
accurate on a different part of the spectrum.

\begin{table}
\begin{tabular}{lll}
\hline\noalign{\smallskip}
\multicolumn{3}{l}{\textbf{Algorithm 1} Computing the eigenvalues using the tropical scaling}\\
\noalign{\smallskip}\hline\noalign{\smallskip}
\multicolumn{3}{l}{\textit{INPUT}: Matrix pencil $P(\lambda)$}\\  		
\multicolumn{3}{l}{\textit{OUTPUT}: List of eigenvalues of $P(\lambda)$}\\
1. & \multicolumn{2}{l}{Compute the corresponding tropical polynomial $\trop p(x)$}\\
2. & \multicolumn{2}{l}{Find the tropical roots of $\trop p(x)$}\\
3. & \multicolumn{2}{l}{For each tropical root such as $\alpha_{i}$ do}\\
& 3.1 & Compute the tropical scaling based on $\alpha_{i}$\\
& 3.2 & Compute the eigenvalues using the QZ algorithm\\
& & and sort them by increasing modulus\\
& 3.3 & Choose the $i$th group of the eigenvalues\\  	
\noalign{\smallskip}\hline
\end{tabular}
\end{table}

To illustrate the algorithm, let $P(\lambda)=A_{0}+A_{1}\lambda+A_{2}\lambda^2$ be a quadratic polynomial matrix and let 
$\trop p(\lambda)=\max(\gamma_0,\gamma_1\lambda,\gamma_2\lambda^2)$ be the tropical polynomial corresponding to this quadratic polynomial matrix.

We refer to the tropical roots of $\trop p(x)$ by $\alpha^+\geq\alpha^-$. If $\alpha^+=\alpha^-$ which happens when $\gamma_{1}^2\leq \gamma_{0}\gamma_{2}$ then, $\alpha=\sqrt{\frac{
\gamma_{0}}{\gamma_{2}}}$ and $\beta=\trop p(\alpha)^{-1}=\gamma_{0}^{-1}$. This case coincides with the scaling of~\cite{vandooren00} in which $
\alpha^*=\sqrt{\frac{\gamma_0}{\gamma_2}}$.

When $\alpha^+\neq\alpha^-$, we will have two different scalings based on $\alpha^+=\frac{\gamma_1}{\gamma_2}$, $\alpha^-=\frac{\gamma_0}{\gamma_1}$ and two different $\beta$ corresponding to the two tropical roots:
\[
\beta^+=\trop p(\alpha^+)^{-1}=\frac{\gamma_{2}}{\gamma_{1}^{2}},
\qquad \beta^-=\trop p(\alpha^-)^{-1}=\frac{1}{\gamma_{0}}.
\]
To compute the eigenvalues of $P(\lambda)$ by using the first companion form linearization, we apply the scaling based on $\alpha^+$, which yields 
\[
\lambda
\begin{pmatrix}
  \frac{1}{\gamma_2}A_2 &  \\
  &I  
\end{pmatrix}
+
\begin{pmatrix}
  \frac{1}{\gamma_1}A_1 & \frac{\gamma_2}{\gamma_1^{2}}A_0 \\
  -I&0  
\end{pmatrix} \enspace ,
\]
to compute the $n$ biggest eigenvalues. 
We apply the scaling based on $\alpha^-$, which yields
\[
\lambda
\begin{pmatrix}
  \frac{\gamma_0}{\gamma_1^2}A_2 &  \\
  &I  
\end{pmatrix}
+
\begin{pmatrix}
  \frac{1}{\gamma_1}A_1 & \frac{1}{\gamma_2}A_0 \\
  -I&0  
\end{pmatrix} \enspace,
\]
to compute the $n$ smallest eigenvalues.

In general, let $\alpha_{1}\leq \alpha_{1}\leq \ldots\leq\alpha_{d}$ be the tropical roots of $\trop p(x)$ counted with multiplicities. 
To compute the $i$th biggest group of eigenvalues, 
we perform the scaling for $\alpha_{i}$, which yields
the following linearization:

\[
\lambda
\begin{pmatrix}
  \beta\alpha_{i}^{d}A_{d} &\\
    & I\\
   & &\ddots \\
  & & &I\\
  & & & &I
\end{pmatrix}
+
\begin{pmatrix}
  \beta\alpha_{i}^{d-1} A_{d-1}&\ldots&\beta\alpha_{i}A_{1}&\beta A_{0}\\
  -I & 0 & \ldots & 0\\
  0 & -I & \ddots & \vdots\\
  \vdots &  &\ddots & 0 \\
  0& \ldots &-I&0
\end{pmatrix}
\]
where $\beta=\trop p(\alpha_i)^{-1}$. 
Doing the same for all the distinct tropical roots, we can compute all the eigenvalues.

\begin{remark}
The interest of Algorithm~1 lies in the accuracy
(since it allows us to solve instances in which the
data have various order of magnitudes).
Its inconvenient
is to call several times
(once for each distinct tropical eigenvalue, and so,
at most $d$ times) the QZ algorithm.
However, we may partition the different tropical
eigenvalues in groups consisting each of eigenvalues of the same order of
magnitude, and then, the speed factor we would loose would be reduced to the number of different groups.
\end{remark}

\section{Splitting of the eigenvalues in tropical groups}
\label{sec:5}
In this section we state a simple theorem concerning the location of the
eigenvalues of a quadratic polynomial matrix, showing that under
a non degeneracy condition, the two tropical roots do provide
the correct estimate of the modulus of the eigenvalues.

We shall need to compare spectra, which may be thought of as
unordered sets, therefore, we define the following
metric (eigenvalue variation), which appeared
in~\cite{Gal08}. We shall use the notation $\spec$
for the spectrum of a matrix or a pencil.
\begin{definition}
\label{eigenvalue variation}
Let $\lambda_1,\ldots \lambda_n $ and $\mu_1\ldots\mu_n $
denote two sequences of complex numbers. The variation between $\lambda$ and
$\mu$ is defined by
\[
v(\lambda,\mu):=\min_{\pi\in S_n}\{\max_{i} |\mu_{\pi(i)}-\lambda_i |\}\enspace,
\]
where $S_n$ is the set of permutations of $\{1,2,\ldots,n \}$.
If $A,B\in\mathbb{C}^{n\times n}$, the eigenvalue variation of $A$ and $B$ is defined by $v(A,B):=v(\spec A,\spec B)$.
\end{definition}
Recall that the quantity $v(\lambda,\mu)$ can be computed
in polynomial time by solving a bottleneck assignment problem.

We shall need the following theorem of Bathia, Elsner, and Krause~\cite{bhatia}.
\begin{theorem}[{\cite{bhatia}}]
\label{perturbation_theorem}
Let $A,B\in\mathbb{C}^{n\times n}$. Then $ v(A,B)\leq 4\times 2^{-1/n}(\|A\|+\|B\|)^{1-1/n}\|A-B\|^{1/n} $ .
\end{theorem}

The following result shows that when the parameter $\delta$ measuring
the separation between the two tropical roots is sufficiently large, and
when the matrices $A_2,A_1$ are well conditioned, then, there are
precisely $n$ eigenvalues of the order of the maximal tropical root. 
By applying the same result to the reciprocal
pencil, we deduce, under the same separation condition,
that when $A_1,A_0$ are well conditioned, there are precisely $n$ eigenvalues
of the order of the minimal tropical root. So, under such conditions, 
the tropical
roots provide accurate a priori estimates of the order
of the eigenvalues of the pencil.

\begin{theorem}[Tropical splitting of eigenvalues]
\label{theorem_location}
Let $P(\lambda)=\lambda^2 A_2+ \lambda A_1+A_0$ where $A_i\in\mathbb{C}^{n\times n}$, and $\gamma_i:=\|A_i\|$, $i=0,1,2$.
Assume that the max-times polynomial
$p(\lambda)=\max(\lambda^2\gamma_2,\lambda\gamma_1,\gamma_0)$
has two distinct tropical roots, $\alpha^+:=\gamma_1/\gamma_2$
and $\alpha^-=\gamma_0/\gamma_1$, and let $\delta:=\alpha^+/\alpha^-$.
Assume that $A_2$ is invertible. Let $\xi_1,\ldots,\xi_n$ denote the
eigenvalues of the pencil $\lambda A_2+A_1$, and let us set
$\xi_{n+1}=\cdots =\xi_{2n}=0$. Then,
\[
v(\spec P,\xi) \leq \frac{C \alpha^+}{\delta^{1/2n}} \enspace ,
\]
where
\[
C:=4\times 2^{-1/2n} \big(2+2\cond A_2+\frac{\cond A_2}{\delta}\big)^{1-1/2n}\big(\cond A_2\big)^{1/2n} \enspace ,
\]
and
\begin{align}
\alpha^+(\cond A_1)^{-1}\leq |\xi_i|\leq \alpha^+\cond A_2 ,\qquad 1\leq i\leq n \enspace .\label{ineq-simple}
\end{align}
\end{theorem}

\begin{proof}
Let us make the scaling corresponding to the maximal tropical root
$\alpha^+=\gamma_1/\gamma_2$, with $\beta^+=\gamma_2/\gamma_1^2$, 
which amounts to considering 
the new polynomial matrix $Q(\mu)=\beta^+P(\alpha^+\mu)=
\bar A_2\mu^2+\bar A_1\mu+\bar A_0$ where
\[
\bar A_2= \gamma_2^{-1}A_2,\qquad \bar A_1=\gamma_1^{-1}A_1,
\qquad \bar A_0 = \frac{\gamma_2}{\gamma_1^2}A_0 \enspace .
\]
Since $A_2$ is invertible, $\lambda$ is an eigenvalue of the pencil
$P$ if and only if $\lambda=\alpha^+\mu$ where $\mu$ is an eigenvalue
of the matrix:
\[ X=
\begin{pmatrix} 
-\bar A_2^{-1}\bar A_1&-\bar A_2^{-1}\bar A_0\\
I&0
\end{pmatrix}
\]
Let $\mu_i,i=1,\ldots,2n$ denote the eigenvalues of this matrix.
Consider 
\[ Y=
\begin{pmatrix} 
-\bar A_2^{-1}\bar A_1&0\\
I&0
\end{pmatrix}
\]
Observe that $\|\bar A_1\|=1$ and
$\|\bar A_0\|= \gamma_2\gamma_0/\gamma_1^2=1/\delta$.
Since the induced Euclidean norm $\|\cdot\|$ is an algebra
norm, we get
\begin{align*}
\|X\|\leq \|I\|+\|\bar A_2^{-1}\bar A_1\|+\|\bar A_2^{-1}\bar A_0\|
& 
\leq 1 + \|A_2^{-1}\|\|A_2\|+ \|A_2^{-1}\|\|A_2\|\|\bar A_0\|\\
& = 1+\cond A_2(1+1/\delta) \enspace .
\end{align*}
Moreover,
\[ 
\|Y\|\leq 1+ \cond A_2 \enspace ,
\qquad 
\|X-Y\|=(\cond A_2)/\delta.
\]
Using Theorem~\ref{perturbation_theorem}, we deduce that
\[
v(\spec X,\spec Y)\leq C/\delta^{1/2n} \enspace .
\]
Since the family of eigenvalues of $P$ coincide with $\alpha^+(\spec X)$, and since the family of numbers $\xi_i$ coincides
with $\alpha^+(\spec Y)$, the first part of the result is proved.

If $\xi$ is an eigenvalue of $A_2\lambda +A_1$, then,
we can write $\xi=\alpha^+\zeta$, where $\zeta$ is an eigenvalue
of $\bar A_2 \mu +\bar A_1$. We deduce that $|\zeta|
\leq \|\bar A_2^{-1}\|\|\bar A_1\|=\cond A_2$, which establishes the second
inequality in~\eqref{ineq-simple}. The first inequality is established along the same
lines, by considering the reciprocal pencil of $\bar A_2 \mu +\bar A_1$.
\qed
\end{proof}
\begin{remark}
Theorem~\ref{theorem_location} is a typical, but special instance
of a general class of results that we discuss in a further work.  
In particular, this theorem can be extended to matrix polynomials
of an arbitrary degree, with a different proof technique. Indeed, the 
idea of the proof above works only for the two ``extreme'' groups of
eigenvalues, whereas in the degree $d$ case, the eigenvalues are split
in $d$ groups (still under nondegeneracy conditions). 
Note also that the exponent in $\delta^{1/2n}$ is suboptimal
\end{remark}
\begin{remark}\label{rk-eig}
In~\cite{abg04,abg04b}, the {\em tropical eigenvalues} are defined as follows. 
The {\em permanent} of a $n\times n$ matrix $B=(b_{ij})$ with entries in $\rmax$ is defined by
\[
\operatorname{per} B:=\max_{\sigma\in S_n} \sum_{1\leq i\leq n}b_{i\sigma(i)} \enspace .
\]
This is
nothing than the value of the optimal assignment problem with weights $(b_{ij})$. The {\em characteristic polynomial} of a matrix $C=(c_{ij})$ is defined 
as the map from $\rmax$ to itself,
\[
x\mapsto P_C(x):=\operatorname{per} (C \oplus xI) \enspace,
\]
where $I$ is the max-plus identity matrix, with diagonal entries
equal to $0$ and off-diagonal entries equal to $-\infty$. The sum
$C\oplus xI$ is interpreted in the max-plus sense, so
\[
(C\oplus xI)_{ij} = \begin{cases} c_{ij} & \text{if $i\neq j$}\\
\max(c_{ii},x) & \text{if $i=j$.}
\end{cases}
\]
The {\em tropical eigenvalues} are defined as the roots of the characteristic
polynomial. The previous definition has an obvious generalization
to the case of tropical matrix polynomials: if $C_0,\ldots,C_d$ are $n\times n$
matrices with entries in $\rmax$, the eigenvalues of the matrix polynomial $C(x):=C_0\oplus C_1x \oplus \cdots \oplus C_d x^d$ are defined
as the roots of the polynomial function $x\mapsto \operatorname{per} (C(x))$.
The roots of this function can be computed in polynomial
time by $O(nd)$ calls to an optimal assignment solver
(the case in which $C(x)=C_0\oplus xI$ was solved
by Burkard and Butkovi\v{c}~\cite{bb02};
the generalization to the degree $d$ case was pointed out
in~\cite{abg04}). When the matrices $A_0,\ldots,A_d$ are scalars, the logarithms of the
tropical roots considered in the present paper are readily seen to coincide
with the tropical eigenvalues of the pencil in which
$C_k$ is the logarithm of the modulus of $A_k$, for $0\leq k\leq d$.
When these matrices are not scalars,
in view of the asymptotic results of~\cite{abg04},
the exponentials of the tropical eigenvalues are expected to provide
more accurate estimates of the moduli of the complex roots.
This alternative approach is the object of a further work,
however, the comparative interest of the tropical roots considered here lies in their simplicity: they only
depend on the norms of $A_0,\ldots,A_d$, and can be computed in linear
time from these norms. They can also be used as a measure of ill-posedness
of the problem (when the tropical roots have different orders of magnitude,
the standard methods in general fail).
\end{remark}

\section{Experimental Results}
\label{sec:6}

\subsection{Quadratic Polynomial Matrices}
\label{sec:61}Consider first $P(\lambda)=A_{0}+A_{1}\lambda+A_{2}\lambda^2$ and its linearization $L=\lambda X + Y $. 
Let $z$ be the eigenvector computed by applying the QZ algorithm to this linearization. 
Both $\zeta_1=z(1:n)$ and $\zeta_2=z(n+1:2n)$ are eigenvectors of $P(\lambda)$. We present our results for both of these eigenvectors;
$\eta_s$ denotes the normwise backward error for the scaling of~\cite{vandooren00}, and $\eta_t$ denotes the same quantity for the tropical scaling.

Our first example coincides with Example~3 of~\cite{vandooren00} where $\|A_2\|_2\approx 5.54\times 10^{-5},\|A_1\|_2\approx 4.73\times10^3,\|A_0\|_2 \approx 6.01\times 10^{-3}$ and $A_i\in\mathbb{C}^{10\times 10}$.
We used $100$ randomly generated pencils normalized to get the mentioned norms and we computed the average of the quantities mentioned in the following table for these pencils. Here we present the results for the $5$ smallest eigenvalues, however for all the eigenvalues, the backward error computed by using the tropical scaling is of order $10^{-16}$ which is the precision of the computation. 
The computations were carried out in SCILAB 4.1.2.
 
\begin{center}
\begin{tabular}{l@{\extracolsep{1mm}}||l|l||l|l||l|l}
$|\lambda|$ & $\eta(\zeta_1,\lambda)$ & $\eta(\zeta_2,\lambda)$ & $\eta_s(\zeta_1,\lambda)$ & $\eta_s(\zeta_2,\lambda)$ & $\eta_t(\zeta_1,\lambda)$ & $\eta_t(\zeta_2,\lambda)$\\
\hline
2.98E-07  &    1.01E-06  & 4.13E-08  & 5.66E-09  & 5.27E-10  &   6.99E-16  & 1.90E-16  \\
5.18E-07  &  1.37E-07    & 3.84E-08  & 8.48E-10  & 4.59E-10  &   2.72E-16  & 1.83E-16  \\
7.38E-07  & 5.81E-08     & 2.92E-08  & 4.59E-10  & 3.91E-10  &   2.31E-16  & 1.71E-16  \\
9.53E-07  &  3.79E-08    & 2.31E-08  & 3.47E-10  & 3.36E-10  &   2.08E-16  & 1.63E-16  \\
1.24E-06  &  3.26E-08    & 2.64E-08  & 3.00E-10  & 3.23E-10  &   1.98E-16  & 1.74E-16  
\end{tabular}
\end{center}
In the second example, we consider a matrix pencil with $\|A_2\|_2\approx 10^{-6},\|A_1\|_2\approx 10^3 ,\|A_0\|_2\approx 10^5$ and $A_i\in\mathbb{C}^{40\times 40}$. Again, we use $100$ randomly generated pencils with the mentioned norms and we compute the average of all the quantities presented in the next table. 
We present the results for the $5$ smallest eigenvalues. 
This time, the computations shown are from MATLAB 7.3.0,
actually, the results are insensitive to this choice, since 
the versions of MATLAB and SCILAB we used both rely on 
the QZ algorithm of Lapack library (version 3.0).
\begin{center}
\begin{tabular}{l@{\extracolsep{1mm}}||l|l||l|l||l|l}
$|\lambda|$ & $\eta(\zeta_1,\lambda)$ & $\eta(\zeta_2,\lambda)$ & $\eta_s(\zeta_1,\lambda)$ & $\eta_s(\zeta_2,\lambda)$ & $\eta_T(\zeta_1,\lambda)$ & $\eta_T(\zeta_2,\lambda)$\\
\hline
1.08E+01 &	2.13E-13&	4.97E-15&	8.98E-12&	4.19E-13&	5.37E-15&	3.99E-16 \\
1.75E+01&	5.20E-14&	4.85E-15&	7.71E-13&	4.09E-13&	6.76E-16&	3.95E-16 \\
2.35E+01&	4.56E-14&	5.25E-15&	6.02E-13&	4.01E-13&	5.54E-16&	3.66E-16 \\
2.93E+01&	4.18E-14&	5.99E-15&	5.03E-13&	3.97E-13&	4.80E-16&	3.47E-16 \\
3.33E+01&	3.77E-14&	5.28E-15&	4.52E-13&	3.84E-13&	4.67E-16&	3.53E-16 
\end{tabular}
\end{center}

\subsection{Polynomial Matrices of Degree $d$}
\label{sec:62}
Consider now the polynomial matrix $P(\lambda)=A_{0}+A_{1}\lambda+\cdots+A_{d}\lambda^d$, and let $L=\lambda X+Y$ be the first companion form linearization of this pencil. If $z$ is an eigenvector for $L$ then $ \zeta_{1}= z(1:n) $ is an eigenvector for $P(\lambda)$. In the following computations, we use $\zeta_1$ to compute the normwise backward error of Matrix pencil, however this is possible to use any $z(kn+1:n(k+1))$ for $k=0\ldots d-1$. 

To illustrate our results, we apply the algorithm for $20$ different randomly generated matrix pencils and then compute the backward error for a specific eigenvalue of these matrix pencils. The 20 values x-axis, in Fig.~\ref{n20d5} and~\ref{n8d10}, identify the random instance while the y-axis shows the $\log_{10}$ of backward error for a specific eigenvalue. Also we sort the eigenvalues in a decreasing order of their absolute value, so $\lambda_1$ is the maximum eigenvalue.

We firstly consider the randomly generated matrix pencils of degree $5$ where the order of magnitude of the Euclidean norm of $A_i$ is as follows:  
\begin{center}
\begin{tabular}{l@{\extracolsep{1mm}}|l|l|l|l|l}
$ \|A_0 \| $ & $ \|A_1 \| $ & $ \|A_2 \| $ & $ \|A_3 \| $ & $ \|A_4 \| $ & $ \|A_5 \| $ 
\\\hline
$ O(10^{-3}) $ & $ O(10^{2})$ & $O(10^{2})$ & $O(10^{-1})$ & $O(10^{-4})$ & $O(10^{5})$ 
\end{tabular}
\end{center}
Fig.~\ref{n20d5} shows the results for this case where the dotted line shows the backward error without scaling and the solid line shows the backward error using the tropical scaling. 
We show the results for the minimum eigenvalue, the ``central'' $50^{\text{th}}$ eigenvalue and the maximum one from top to down. 
In particular, the picture at the top
shows a dramatic improvement since the smallest of the eigenvalues is
not computed accurately
(backward error almost of order one) without the scaling, whereas for
the biggest of the eigenvalues, the scaling typically improves the backward
error by a factor 10. For the central eigenvalue,
the improvement we get is intermediate.
The second example concerns the randomly generated matrix pencil with degree $10$ while the order of the norm of the coefficient matrices are as follows:
\begin{center}
\begin{tabular}{l@{\extracolsep{1mm}}|l|l|l|l|l}
$ \|A_0 \| $ & $ \|A_1 \| $ & $ \|A_2 \| $ & $ \|A_3 \| $ & $ \|A_4 \| $ & $ \|A_5 \| $\\
\hline
 $O(10^{-5})$ & $O(10^{-2})$ & $O(10^{-3})$ & $O(10^{-4})$ & $O(10^{2})$ & $O(1) $\end{tabular}
\begin{tabular}{l@{\extracolsep{1mm}}|l|l|l|l}
$ \|A_6 \| $ & $ \|A_7 \| $ & $ \|A_8 \| $ & $ \|A_9 \| $ & $ \|A_{10} \| $\\
\hline 
$O(10^{3})$ & $O(10^{-3})$ & $O(10^{4})$ & $O(10^{2})$ & $O(10^{5})$ 
\end{tabular}
\end{center}
In this example, the order of the norms differ from $10^{-5}$ to $10^5$ and the space dimension of $A_i$ is $8$. Figure ~\ref{n8d10} shows the results for this case where the dotted line shows the backward error without scaling and the solid line shows the backward error using tropical scaling. Again we show the results for the minimum eigenvalue, the $40$th eigenvalue and the maximum one from top to down.

\begin{figure}[htbp]
\begin{minipage}[l]{.49\linewidth}
\centering
\includegraphics[scale=0.3]{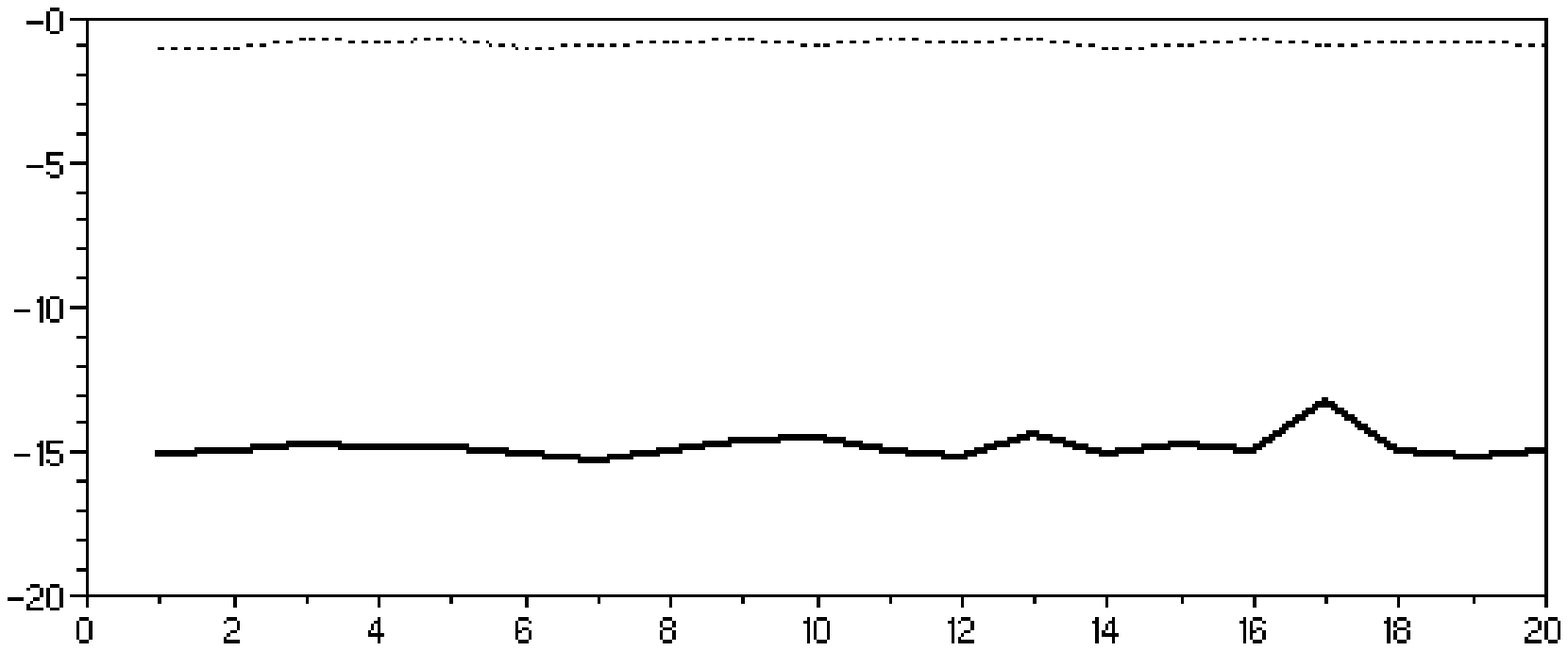} 
 \includegraphics[scale=0.3]{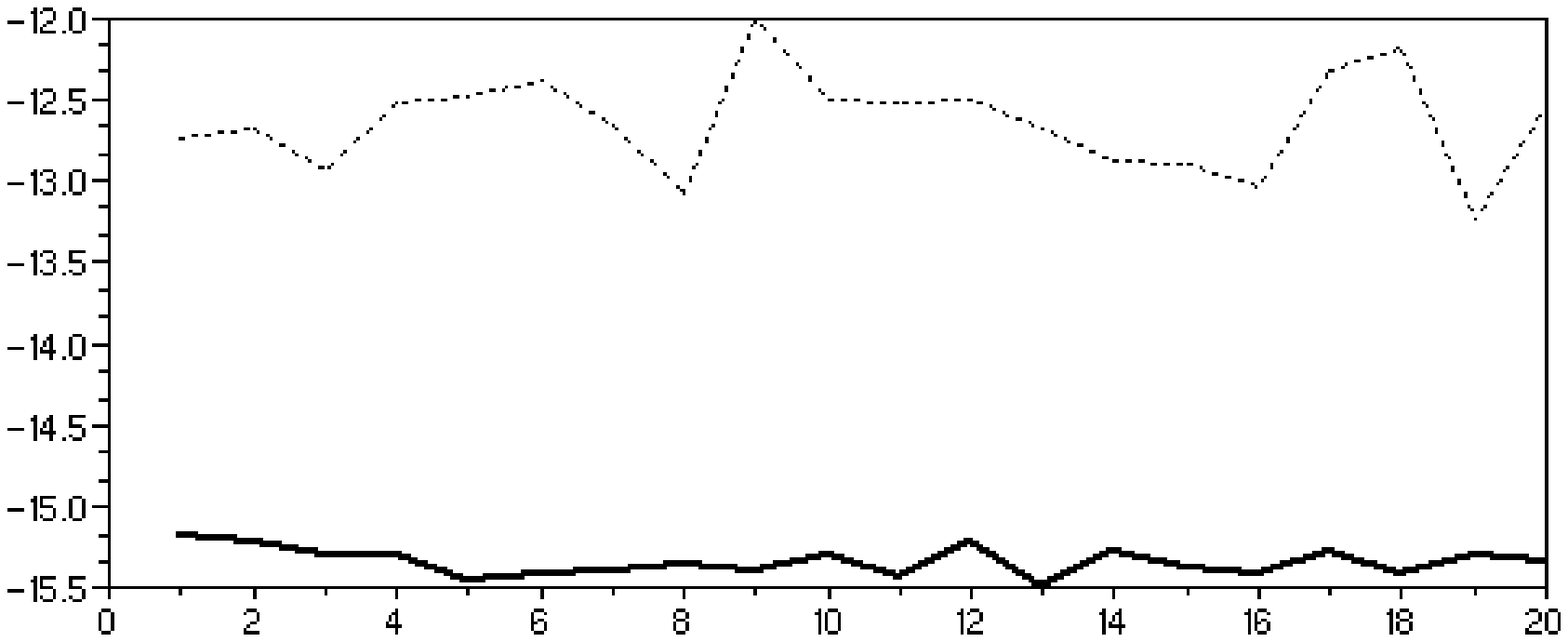}
 \includegraphics[scale=0.3]{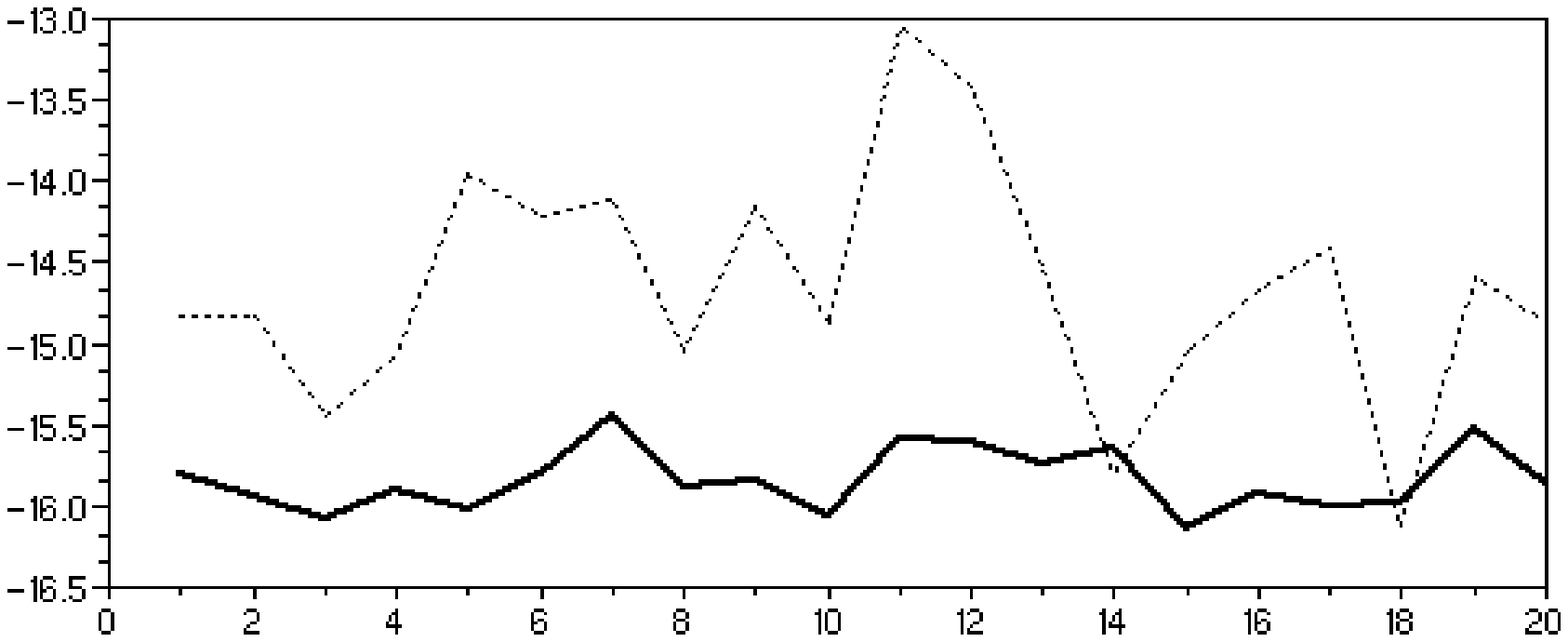}

 \caption{Backward error for randomly generated matrix pencils with $n=20$, $d=5$. }
 \label{n20d5}
\end{minipage} 
\begin{minipage}[r]{.45\linewidth}
\centering
\includegraphics[scale=0.3]{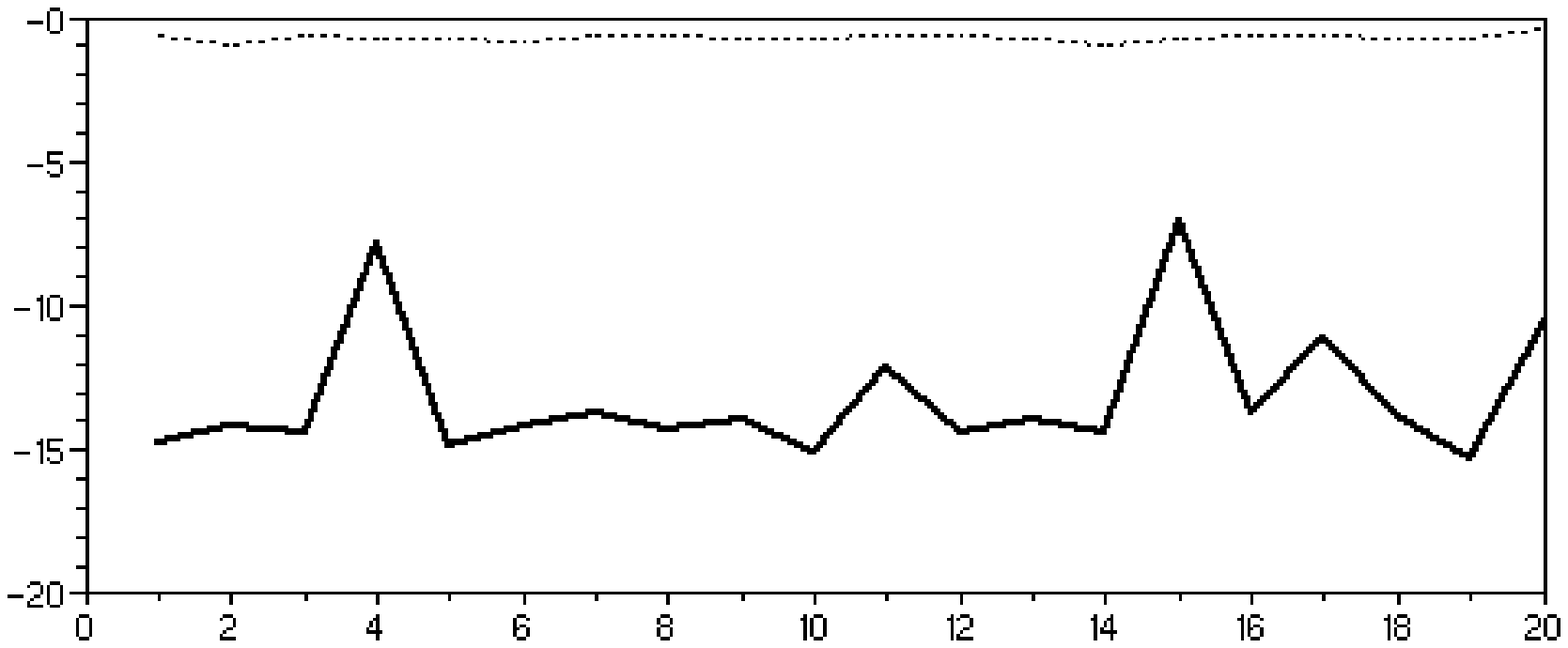}
 \includegraphics[scale=0.3]{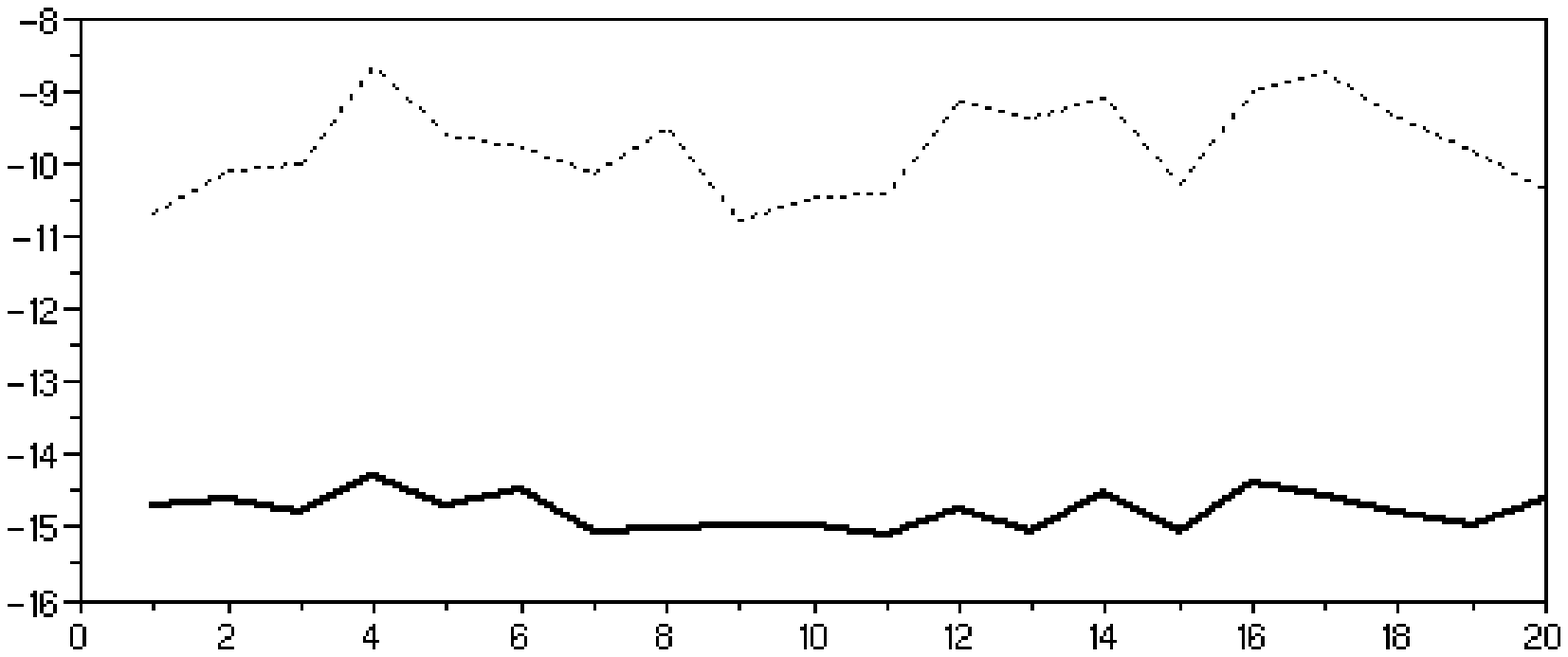}
  \includegraphics[scale=0.3]{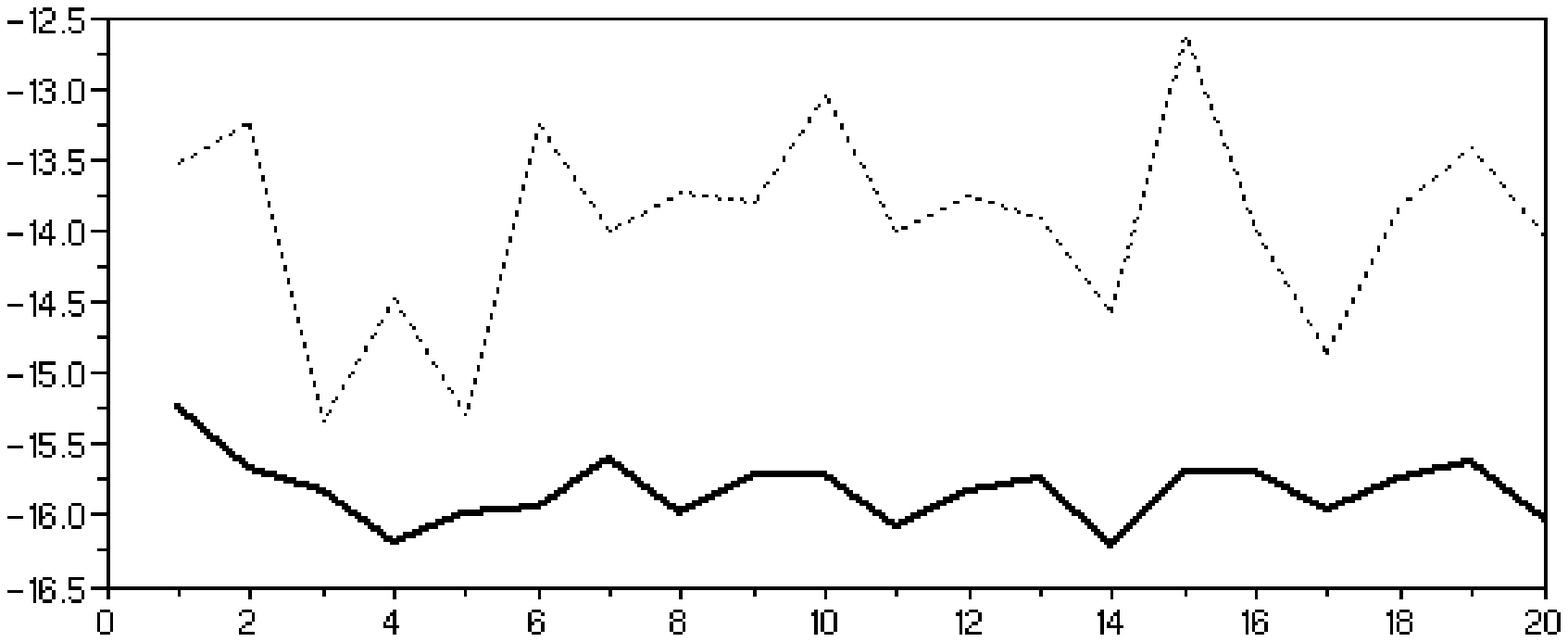}

  \caption{Backward error for randomly generated matrix pencils with $n=8$, $d=10$.}
  \label{n8d10}
\end{minipage}    
\end{figure}


\end{document}